\documentclass[10pt,a4paper]{article}
\title{Reciprocity of the Wigner derivative for spherical tetrahedra}
\author{Bruce Bartlett and V. Hosana Ranaivomanana}
\date{}
\usepackage[latin1]{inputenc}
\usepackage{tikzit}
\usepackage{amsthm}
\usepackage{amsmath}
\usepackage{amsfonts}
\usepackage{amssymb}
\usepackage{graphicx}
\usepackage{bm}
\usepackage[all]{xy}
\usepackage{caption}
\usepackage{subcaption}
\usepackage{cite}
\newtheorem{thm}{{Theorem}}[section]
\newtheorem{lem}[thm]{Lemma}
\newtheorem{cor}[thm]{Corollary}

\newtheorem{rem}[thm]{Remark}

\tikzstyle{red dot}=[fill=red, draw=black, shape=circle, tikzit fill=red]
\tikzstyle{green dot}=[fill={rgb,255: red,191; green,255; blue,0}, draw=black, shape=circle, tikzit fill={rgb,255: red,191; green,255; blue,0}, tikzit draw=black]
\tikzstyle{blue text}=[text=blue]
\tikzstyle{red text}=[red]
\tikzstyle{green text}=[text=green]

\tikzstyle{basic dashed}=[-, dashed]
\tikzstyle{basic blue}=[-, draw=blue]
\tikzstyle{basic blue dashed}=[-, draw=blue, dashed]
\tikzstyle{basic green}=[-, draw=green]
\tikzstyle{basic green dashed}=[-, draw=green, dashed]
\tikzstyle{basic green dotted}=[-, draw=green, dotted]
\tikzstyle{simple}=[-]

\DeclareMathOperator{\Rep}{Rep}
\DeclareMathOperator{\Lk}{Lk}

\begin{document}
	\maketitle
\begin{abstract}
	The Wigner derivative is the partial derivative of dihedral angle with respect to opposite edge length in a tetrahedron, all other edge lengths remaining fixed. We compute the inverse Wigner derivative for spherical tetrahedra, namely the partial derivative of edge length with respect to opposite dihedral angle, all other dihedral angles remaining fixed. We show that the inverse Wigner derivative is actually equal to the Wigner derivative. These computations are motivated by the asymptotics of the classical and quantum 6j symbols for SU(2).
\end{abstract}
\section{Introduction}
In his seminal book on group theory and quantum mechanics from 1959 \cite{wigner2012group}, Wigner studied the classical 6j symbols for $SU(2)$, which encode the associator data \cite{roberts1999classical} for $\Rep SU(2)$, its tensor category of representations. He related the 6j symbol 
\begin{equation} \label{6jsymbol}
\begin{Bmatrix}
J&j_2&j'\\
j_1&j_3&j
\end{Bmatrix}
\end{equation}
to a Euclidean tetrahedron with side lengths given by $j_1$, $j_2$, $j_3$, $J$, $j$ and $j'$ and gave a heuristic argument that the square of this 6j symbol should (on average, for large spins) be proportional to the partial derivative $\frac{\partial \theta}{\partial j'}$ of the dihedral angle $\theta$ at edge $j$ with respect to the length of the opposite edge $j'$, all other lengths being held fixed (see Figure \ref{Wigner-pic}).

\begin{figure}[t]
\centering
\begin{subfigure}[b]{0.25\linewidth}
	\includegraphics[width=\linewidth]{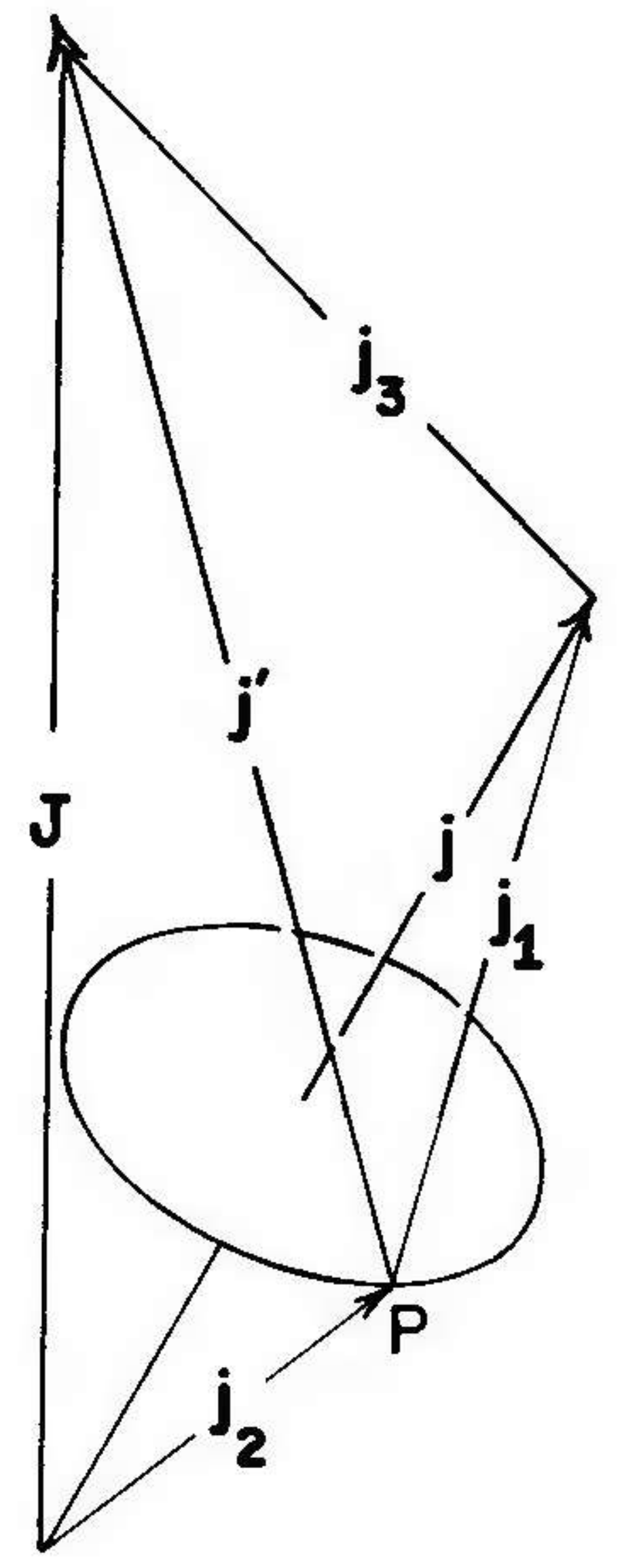}
	\caption{\label{Wigner-pic}}
\end{subfigure}\qquad
\begin{subfigure}[b]{0.45\linewidth}
	\begin{tikzpicture}
	\begin{pgfonlayer}{nodelayer}
		\node [style=none] (0) at (0.5, -4.5) {};
		\node [style=none] (1) at (4.5, -1.5) {};
		\node [style=none] (2) at (-2.25, 5.25) {};
		\node [style=none] (3) at (-5.25, -3) {};
		\node [style=none] (4) at (-0.325, 1.275) {};
		\node [style=none] (5) at (-1.35, 1.625) {};
		\node [style=none] (6) at (0.375, 1.9) {};
		\node [style=none] (7) at (-0.975, 1.5) {};
		\node [style=none] (8) at (0.15, 1.725) {};
		\node [style=none] (9) at (-0.225, 2.25) {$\theta$};
		\node [style=none] (10) at (-1.5, -1.25) {$l'$};
		\node [style=none] (12) at (0.325, 0.125) {$l$};
	\end{pgfonlayer}
	\begin{pgfonlayer}{edgelayer}
		\draw [bend right=15] (2.center) to (3.center);
		\draw [in=-180, out=-30, looseness=0.75] (3.center) to (0.center);
		\draw [in=-135, out=30, looseness=1.25] (0.center) to (1.center);
		\draw [bend left=45] (2.center) to (1.center);
		\draw [bend left=15] (2.center) to (0.center);
		\draw [style=basic dashed, bend left=15] (3.center) to (1.center);
		\draw (5.center) to (4.center);
		\draw (6.center) to (4.center);
		\draw [style=basic dashed, bend left] (7.center) to (8.center);
	\end{pgfonlayer}
\end{tikzpicture}
	\caption{\label{Wigner-derivative}}
\end{subfigure}
\caption{(a) The Euclidean tetrahedron corresponding to the 6j symbol in equation \eqref{6jsymbol}. If the lengths $j_1$, $j_2$, $j_3$, $j$ and $J$ are held constant, then $P$ can still traverse the indicated circle, changing $j'$. The probability of a given tetrahedron occurring is proportional to $\frac{\partial \theta}{\partial j'}$ where $\theta$ is the dihedral angle at the edge with length $j$. Taken from \cite{wigner2012group}; see also \cite{biedenharn1981racah}. (b) The Wigner derivative for a spherical tetrahedron is $\frac{\partial \theta}{\partial l'}$, the partial derivative of dihedral angle with respect to opposite edge length, all other edge lengths held fixed. }
\end{figure}
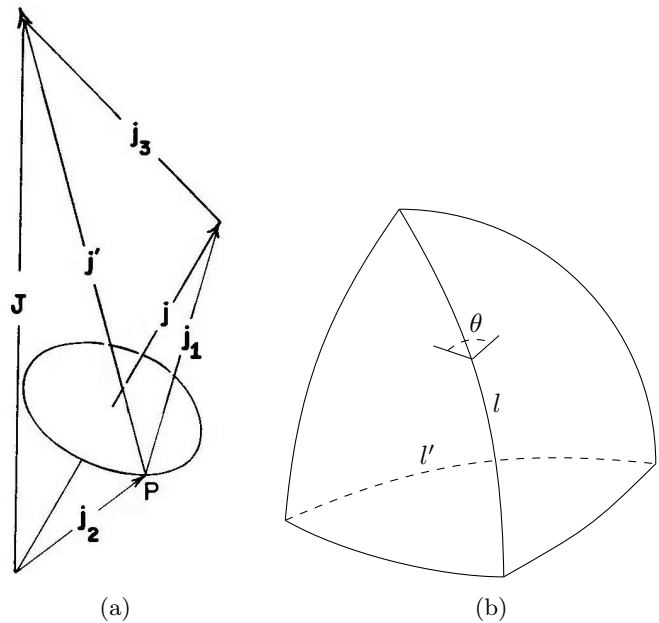

In 1968 the physicists Ponzano and Regge conjectured a more refined formula for the asymptotics of the classical 6j symbols, which included an oscillatory term. This formula was first proved rigorously by Roberts in 1999, using geometric quantization techniques \cite{roberts1999classical}, and since then a number of other proofs have been given \cite{barrett2003asymptotics, freidel2003asymptotics, aquilanti2012semiclassical, garoufalidis2013asymptotics, costantino2015generating}.

In 2003 Taylor and Woodward gave a corresponding asymptotic formula for the {\em quantum} 6j symbols, relating their asymptotics to the geometry of {\em spherical} tetrahedra \cite{taylor20036j, taylor2003quantum}.
 In their outline of a possible geometric proof of their formula (this approach was later made rigorous by March\'e and Paul\cite{marche2015toeplitz}), the partial derivative of dihedral angle with respect to opposite edge length (this time for a spherical tetrahedron) again played a crucial role. Following Taylor, we call this the {\em Wigner derivative} (see Figure \ref{Wigner-derivative}).	

Given a spherical tetrahedron with vertices $v_0,v_1,v_2,v_3$ and edge lengths $l_{ij}$, let $G$ be the length Gram matrix, $G_{ij}=\cos (l_{ij})$. Taylor and Woodward's formula for the Wigner derivative is as follows\footnote{Note that the actual statement of Proposition 2.4.1.(n) in \cite{taylor20036j} contains a typo. The left hand side should be $\frac{\partial \theta_{ab}}{\partial l_{cd}}$ not $(\frac{\partial \theta_{ab}}{\partial l_{cd}})^{-1}$.}. (We will give an independent proof in Section \ref{Reciprocity-section}).
  \begin{thm}[Taylor-Woodward \cite{taylor20036j}]
	The Wigner derivative for a spherical tetrahedron is
	\begin{equation}
	\label{1}
	\frac{\partial \theta(l_{ij})}{\partial l'}=\frac{\sin l\sin l'}{\sqrt{\det G}}
	\end{equation} 
	where $\theta$ is the interior dihedral angle at the edge with length $l$ and $l'$is the length of the opposite edge (see Figure \ref{Wigner-derivative}).
   \end{thm}
Unlike a Euclidean tetrahedron, a spherical tetrahedron is determined up to isometry by its six edge lengths as well as by its six dihedral angles. So there is a 1-1 correspondence between edge lengths and dihedral angles
$$
(l_{01},l_{02},l_{03},l_{12},l_{13},l_{23})\leftrightarrow (\theta_{01},\theta_{02},\theta_{03},\theta_{12},\theta_{13},\theta_{23}).
$$
See \eqref{8} in Section \ref{Reciprocity-section} for an explicit formula. Therefore it makes sense to ask about the {\em inverse} Jacobian matrix $\frac{\partial l_{ij}}{\partial \theta_{kl}}$ and in particular the {\em inverse} Wigner derivative $\frac{\partial l'}{\partial \theta}$ in Figure \ref{Wigner-derivative}. Indeed, in our work we were led to consider this inverse Jacobian as it shows up in the stationary phase approximation for a conjectural integral formula for the quantum 6j symbols. The main result of this paper is as follows.
\begin{thm}
	\label{inv Wigner}
	The inverse Wigner derivative for a spherical tetrahedron (see Figure \ref{Wigner-derivative}) is
	\begin{equation}
	\label{2}
	 \frac{\partial l'(\theta_{ij})}{\partial \theta}=\frac{\sin l\sin l'}{\sqrt{\det G}} 
	 \end{equation}
\end{thm}
Comparing with formula \eqref{1} for the Wigner derivative, we obtain the following corollary.
\begin{cor}[Reciprocity of the Wigner derivative]
	For spherical tetrahedra, the Wigner derivative and the inverse Wigner derivative are equal:
	\begin{equation}
	\label{3}
	\frac{\partial \theta(l_{ij})}{\partial l'}=\frac{\partial l'(\theta_{ij})}{\partial \theta}
	\end{equation}
	
\end{cor}
\begin{rem}
	In the proof of \cite[Proposition 2.4.1.(n)]{taylor20036j} and in \cite[Proposition 2.2.0.5]{taylor2003quantum} Taylor and Woodward show that
\[
	\label{4}
	\frac{\partial l'}{\partial \theta}=\frac{\sqrt{\det G}}{\sin l\sin l'}
\]
which is the reciprocal of our formula \eqref{2} in Theorem \ref{inv Wigner} and thus seems to contradict it. What is going on? The answer is that they are {\em different} partial derivatives as different sets of variables are being held constant. In our formula \eqref{2}, $\theta$ is changing while keeping the five remaining dihedral angles constant, while in Taylor and Woodward's formula \eqref{3}, $\theta$ is changing while all lengths excluding $l'$ are being held constant. It is interesting that these two different partial derivatives are reciprocals of each other. To the best of our knowledge, formula \eqref{2} and its corollary \eqref{3} are new (see \cite{petrera2014spherical} for related work).
\end{rem}
	\paragraph{Overview of paper}
	
	In Section \ref{section 1} we show, as a warm-up result, that reciprocity of the Wigner derivative holds for spherical triangles. In Section \ref{Reciprocity-section} we consider spherical tetrahedra. By relating the dihedral angles to the edge lengths via the {\em links} of the vertices, we can apply the reasoning from Section 1 and hence prove our main results, Theorem \ref{main-theorem} and Corollary \ref{main-corollary}.

\section{Reciprocity of the Wigner derivative for spherical triangles}
\label{section 1}
In this section we review some elementary spherical trigonometry, and prove reciprocity of the Wigner derivative for spherical triangles. This serves as a warm-up example before tackling spherical tetrahedra.

Consider a spherical triangle $\Delta\subseteq S^2$ as in Figure \ref{spherical-triangle-fig} with vertices $v_0,v_1,v_2\in S^2$,
$$
\Delta := \{t_0v_0+t_1v_1+t_2v_2: t_0,t_1,t_2\geq 0\}\cap S^2.
$$

\begin{figure}[t]
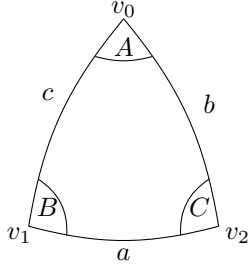

\ctikzfig{sphericaltriangle}
\caption{\label{spherical-triangle-fig} A spherical triangle.}
\end{figure}

The {\em sine law} says that 
\begin{equation}
\label{5}
\frac{\sin a}{\sin A}=\frac{\sin b}{\sin B}=\frac{\sin c} {\sin C} \, .
\end{equation}
The {\em cosine law} expresses the interior angles in terms of the side lengths,
\begin{equation}
\label{cosA}
\cos A=\frac{\cos a-\cos b\cos c}{\sin b\sin c} \, ,
\end{equation}
while the {\em dual cosine law} expresses the side lengths in terms of the interior angles:
\begin{equation}
\label{cosa}
\cos a =\frac{\cos A+\cos B\cos C}{\sin B\sin C} 
\end{equation} 
For the sine and cosine laws, see \cite{cosine-law-wiki}. From (\ref{cosA}) and (\ref{cosa}) we obtain 
\begin{equation}
\label{6}
\frac{\partial A}{\partial a}=\frac{\sin a}{\sin A\sin b\sin c}, \qquad \frac{\partial a}{\partial A}=\frac{\sin A}{\sin a\sin B\sin C}
\end{equation}
To write these formulas in the form of (\ref{1}) and (\ref{2}), we introduce the length Gram matrix $G_{ij}=\cos l(v_i,v_j)$,
$$
G=\begin{pmatrix}
1&\cos c&\cos b\\
\cos c&1&\cos a\\
\cos b& \cos a& 1
\end{pmatrix} \, .
$$
The following fact is fairly well known, but we include a proof in order to be self-contained and because Lemma \ref{lemma5} uses similar manipulations.
\begin{lem}
    \label{lemma3}
	$\sqrt{\det G}=\sin A\sin b\sin c$
\end{lem}
\begin{proof}
	By row operations we obtain
	\begin{align*}
	\det G&= \det \begin{pmatrix}
	1&\cos c&\cos b\\
	0&1-\cos^2 c&\cos a-\cos b\cos c\\
	0& \cos a-\cos b\cos c& 1-\cos^2 b
	\end{pmatrix}\\
	&=\sin^2 b\sin^2 c-(\cos a-\cos b\cos c)^2\\
	&= \sin^2 b\sin ^2c - \sin^2 b\sin^2 c\ \left(\frac{\cos a-\cos b\cos c}{\sin b\sin c}\right)^2\\
	&=\sin^2b\sin^2c\ (1-\cos^2 A)\\
	&=\sin^2b\sin^2c\ \sin^2A
	\end{align*}
	where we have used the cosine law in the second last step.
\end{proof}	
This allows us to prove the formula for the Wigner derivative and its inverse for spherical triangles, and show that they are equal.
\begin{thm}[Wigner reciprocity for spherical triangles]
	For spherical triangles, we have
	$$\frac{\partial A(a,b,c)}{\partial a}=\frac{\sin a}{\sqrt{\det G}}=\frac{\partial a(A,B,C)}{\partial A} \, .
	$$
\end{thm}
\begin{proof}
	The first equation follows directly from (\ref{6}a) and Lemma \ref{lemma3}. The second equation follows from: 
	\begin{align*}
	\frac{\frac{\partial A}{\partial a}}{\frac{\partial a}{\partial A}}&=\frac{\sin^2 a \sin B\sin C}{\sin^2 A\sin b\sin c} && \text{(by \ref{6})}\\
	&=1 && \text{(by the sine rule)}
	\end{align*} 
\end{proof}

\section{Reciprocity of the Wigner derivative for spherical tetrahedra}
\label{Reciprocity-section}
In this section we prove reciprocity of the Wigner derivative for spherical tetrahedra. Our method is to use the links of the vertices (as in \cite{luo20083}) to express the dihedral angles as explicit functions of the edge lengths, and then to use Freidel and Louapre's formula \cite{freidel2003asymptotics} for the determinant of the Gram matrix.

Consider a spherical tetrahedron $\Delta\subseteq S^3$ with vertices $v_0,v_1,v_2,v_3\in S^3$,
$$
\Delta := \{t_0v_0+t_1v_1+t_2v_2+t_3v_3:t_0,t_1,t_2,t_3\geq 0\}\cap S^3.
$$
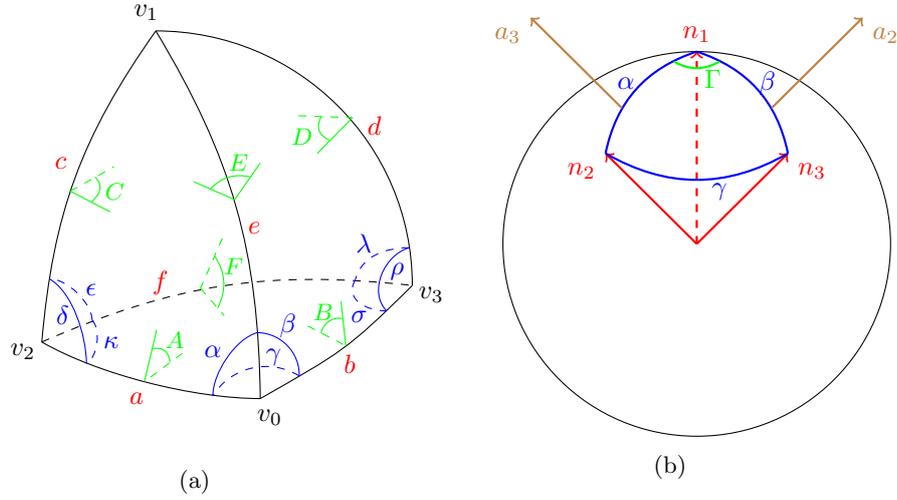
\begin{figure}[t]
\centering
\begin{subfigure}{0.4\linewidth}
	\ctikzfig{tetrahedron}
	\caption{\label{tetrahedron-with-angles}}
\end{subfigure}\qquad \qquad 
\begin{subfigure}{0.4\linewidth}
\begin{tikzpicture}[scale=0.6]
\draw (0,0) circle (4.25cm);
\draw[red,thick,->] (0,0)--(2,2) node[anchor=north west] {$n_3$};
\draw[red, thick,->] (0,0)--(-2,2) node[anchor=north east] {$n_2$};
\draw[red,dashed, thick, ->] (0,0)--(0,4.25) node[anchor=south] {$n_1$};
\draw[blue, thick] (0,4.25) to[bend right] (-2,2);
\draw[blue, thick] (0,4.25) to[bend left] (2,2);
\draw[blue, thick] (-2,2) to[bend right] (2,2);
\draw (-1.55,3.55) node[blue, thick]{$\alpha$};
\draw (1.55,3.55) node[blue, thick] {$\beta$};
\draw (0.5,1.15) node[blue, thick] {$\gamma$};
\draw[green, thick] (-0.5,4.03) to[bend right] (0.5,4.03) ;
\draw (0.35,3.65) node[green, thick] {$\Gamma$};
\draw[brown, thick, ->] (1.64,3)--(3.64,5) node[anchor= north west] {$a_2$};
\draw[brown, thick, ->] (-1.64,3)--(-3.64,5) node[anchor= north east] {$a_3$};
\end{tikzpicture}
	\caption{\label{sphere-pic}}
\end{subfigure}

\caption{\label{tet-fig-big}(a) A spherical tetrahedron $\Delta$. (b)  The link of $v_0$.}
\end{figure}
In Figure \ref{tet-fig-big}a, the edge lengths $a,b,c,d,e,f$ and interior dihedral angles $A,B,C,D,E,F$ are shown, as well as the {\em inner angles} at $v_0,v_2$ and $v_3$. 

The {\em link} $\Lk(v)$ of a vertex $v$ is the spherical triangle with edge lengths given by the inner angles at $v$. In $\Lk(v_0)$, let $\Gamma$ be the interior angle opposite the edge with length $\gamma$, as in Figure \ref{tet-fig-big}b. The following fact is used in \cite{luo20083}; here we give an explicit proof.

\begin{lem}
	$\Gamma =E.$
\end{lem}
\begin{proof}
	By acting with an appropriate element of $SO(4)$, we can rotate $\Delta$ so that
	$$
	v_0=(1,0,0,0),\qquad v_i=(\cos\theta_i,\sin\theta_in_i) \quad i=1,\cdots,3
	$$ 
Here $n_i\in S^2$ are the vertices of $Lk(v_0)$ as in Figure \ref{tet-fig-big}b. By definition, 
$$
\cos E=-w_2 \cdot w_3
$$
where $w_2,w_3\in \mathbb{R}^4$ are the outward unit normals to the faces $v_0v_1v_3$ and $v_0v_1v_2$ of $\Delta$ respectively. Evidently we have
$$
w_2=(0,a_2), \qquad w_3=(0,a_3)
$$
where $a_2,a_3\in \mathbb{R}^3 $ are the outward unit normals to the edges $n_1n_3$ and $n_1n_2$ of $Lk(v_0)$ respectively (see Figure \ref{tet-fig-big}b). But by definition, 
$$
\cos\Gamma=-a_2 \cdot a_3
$$ 
which shows that $\Gamma = E$.
\end{proof}
\begin{lem} \label{lemma4} In the spherical tetrahedron $\Delta$, the Wigner derivative and inverse Wigner derivative are:
	\begin{equation} \label{w1}
	\frac{\partial E(a,b,c,d,e,f)}{\partial f}=\frac{\sin f}{\sin E\sin\alpha\sin\beta\sin a\sin b}
	\end{equation}
	\begin{equation} \label{w2}
	\frac{\partial f(A,B,C,D,E,F)}{\partial E}=\frac{\sin E}{\sin f\sin \kappa\sin\sigma\sin A\sin B}
	\end{equation}
\end{lem}
\begin{proof}
	For the Wigner derivative,
	\begin{align}
	\nonumber E & =E(\alpha ,\beta ,\gamma) && \left(\begin{array}{c} \text{by cosine rule} \\ \text{for $\Lk(v_0)$, see Fig. \ref{figure3}} \end{array} \right) \\
	\label{8}
	&=E(\alpha (a,c,e),\beta (b,d,e),\gamma (a,b,f)) && \left( \begin{array}{c} \text{by cosine rule for triangles} \\ \text{$v_0v_1v_2,v_0v_1v_3$ and $v_0v_2v_3$} \end{array} \right)
	\end{align}
	and so by the chain rule, 
	\begin{align*}
	\frac{\partial E}{\partial f}&=\frac{\partial E}{\partial \gamma}\frac{\partial \gamma}{\partial f}\\
	&=\left(\frac{\sin \gamma}{\sin E\sin \alpha \sin \beta}\right)\left(\frac{\sin f}{\sin \gamma \sin a\sin b}\right)\\
	&=\frac{\sin f}{\sin E\sin \alpha \sin \beta \sin a \sin b} \, .
	\end{align*}
	For the inverse Wigner derivative,
	\begin{align*}
	f &= f(\kappa ,\sigma ,\gamma) && \left( \begin{array}{c} \text{by dual cosine rule} \\ \text{for triangle $v_0v_2v_3$ }\end{array}\right)\\
	&= f(\kappa (A,C,F),\sigma (B,D,F),\gamma (A,B,E)) &&\left( \begin{array}{c} \text{by dual cosine rule for $\Lk(v_2)$,} \\ \text{$Lk(v_3)$ and $Lk(v_0)$, see Fig. \ref{figure3}} \end{array} \right)
	\end{align*}
	and so by the chain rule,
	\begin{align*}
	\frac{\partial f}{\partial E}&=\frac{\partial f}{\partial \gamma}\frac{\partial \gamma}{\partial E}\\
	&= (\frac{\sin \gamma}{\sin f\sin \kappa \sin \sigma})(\frac{\sin E}{\sin \gamma \sin A\sin B})\\
	&= \frac{\sin E}{\sin f\sin \kappa\sin \sigma \sin A\sin B} \, .
	\end{align*}
\end{proof}
\begin{figure}[t]
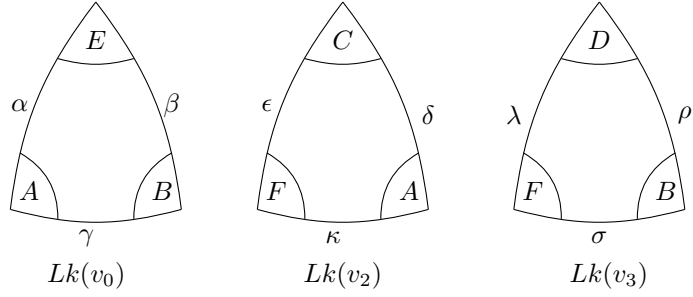

	\ctikzfig{sphericaltriangles}
	\caption{The links of $v_0,v_2$ and $v_3$.}
	\label{figure3}
\end{figure}

To write these derivatives in the form (\ref{1}) and (\ref{2}), we will need Freidel and Louapre's formula for the determinant of the $4\times 4$ Gram matrix $G_{ij}=\cos (l(v_i,v_j))$, for which we give our own proof.
\begin{lem}[See \cite{freidel2003asymptotics}]
	\label{lemma5} $\sqrt{\det G}=\sin a\sin b\sin e\sin \alpha \sin \beta \sin E$.
\end{lem}
\begin{proof}
	\begin{align*}
	\det G&= \det \begin{pmatrix}
	1&\cos e&\cos a&\cos b\\
	\cos e&1&\cos c&\cos d\\
	\cos a&\cos c&1&\cos f\\
	\cos b&\cos d&\cos f&1
	\end{pmatrix}\\
	&= \det \begin{pmatrix}
	1&\cos e&\cos a&\cos b\\
	0&1-\cos^2 e&\cos c-\cos a\cos e&\cos d-\cos b\cos e\\
	0& \cos c-\cos a\cos e& 1-\cos^2 a&\cos f-\cos a\cos b\\
	0& \cos d-\cos b\cos e& \cos f-\cos a\cos b&1-\cos^2b
	\end{pmatrix}\\
	&= \sin^2a \sin^2b\sin^2e \det \begin{pmatrix}
	1& \cos \alpha& \cos \beta\\
	\cos \alpha& 1& \cos \gamma\\
	\cos \beta& \cos \gamma& 1
	\end{pmatrix}\\
	&= \sin^2a \sin^2b\sin^2e \det G'
	\end{align*}
	where $G'$ is the Gram matrix of $Lk(v_0)$. Now use Lemma \ref{lemma3}.
\end{proof}
Note that Lemmas \ref{lemma4} and \ref{lemma5} combine to give a different proof of Taylor and Woodward' s formula for the Wigner derivative.
\begin{thm}[Taylor-Woodward \cite{taylor20036j}]
	The Wigner derivative for a spherical tetrahedron (see Figure \ref{tet-fig-big}a) is $$\frac{\partial E(a,b,c,d,e,f)}{\partial f}=\frac{\sin e\sin f}{\sqrt{\det G}}$$
\end{thm} 
Moreover, we can now prove our main results.
\begin{thm} \label{main-theorem}
	The inverse Wigner derivative for a spherical tetrahedron (see Figure \ref{tet-fig-big}a) is 
	$$\frac{\partial f(A,B,C,D,E,F)}{\partial E}=\frac{\sin e\sin f}{\sqrt{\det G}}$$
\end{thm}
\begin{cor}[Reciprocity of the Wigner derivative] \label{main-corollary}
	For spherical tetrahedra, the Wigner derivative is equal to the inverse Wigner derivative.
\end{cor}
\begin{proof}
	Both results follow from:
	\begin{align*}
	\frac{\frac{\partial E}{\partial f}}{\frac{\partial f}{\partial E}}&=\frac{\sin^2 f\sin A\sin B\sin \kappa\sin \sigma}{\sin^2 E\sin\beta \sin\alpha\sin b\sin a}\\
	&=\frac{\sin^2f}{\sin^2E}\frac{\sin^2E}{\sin^2\gamma}\frac{\sin^2\gamma}{\sin^2f} && \left( \begin{array}{c} \text{using sine law for $Lk(v_0)$} \\ \text{and triangle $v_0v_2v_3$} \end{array} \right)\\
	&=1
	\end{align*}
\end{proof}
\bibliographystyle{plain}
\bibliography{references.bib}
\end{document}